\documentclass[preprint,12pt]{elsarticle}

\usepackage{amssymb}
\usepackage{amsmath}
\usepackage{amsthm}
\usepackage{mathrsfs}
\usepackage{graphicx}
\usepackage{algorithm,algorithmic}
\usepackage{subcaption}
\usepackage{xcolor}
\usepackage{cite}
\usepackage{stfloats}
\usepackage[hyphens]{url}
\usepackage{hyperref}
\newcommand{\dd}{\; \mathrm{d}}
\usepackage{tikz}
\usetikzlibrary{calc,patterns,decorations.pathmorphing,decorations.markings}
\usetikzlibrary{fit, shapes,arrows, calc, chains, positioning, quotes, shadows}
\tikzstyle{block} = [draw, rectangle, 
    minimum height=3em, minimum width=4em,align=center, rounded corners=0.1cm, fill=gray!0!white]
\tikzstyle{cblock} = [draw, rectangle, 
    minimum height=3em, minimum width=4em,align=center, rounded corners=0.1cm, fill=gray!25!white]
\tikzstyle{circ} = [draw, circle, 
   node distance=3em ,align=center, fill=gray!0!white]
\tikzstyle{ccirc} = [draw, circle, 
   node distance=3em ,align=center, fill=gray!25!white]
\tikzstyle{blueblock} = [draw, rectangle, 
    minimum height=3em, minimum width=4em,align=center, color = blue, rounded corners=0.1cm]
\tikzstyle{blckblock} = [draw, rectangle, 
    minimum height=3em, minimum width=4em,align=center, color = black, rounded corners=0.1cm]    
\tikzstyle{hblock} = [draw, rectangle, 
    minimum height=15em, minimum width=3em,align=center, rounded corners=0.1cm]
\tikzstyle{hcblock} = [draw, rectangle, 
    minimum height=13em, minimum width=3em,align=center, rounded corners=0.1cm, fill=gray!25!white    ]
\tikzstyle{eblock} = [draw, rectangle, 
    minimum height=3em, minimum width=4em,align=center, rounded corners=0.1cm, fill=white, color=white]
\tikzstyle{outerblock} = [draw, rectangle, dashed, inner sep = .85em]
\tikzstyle{sum} = [draw, circle, node distance=1cm]
\tikzstyle{var} = [draw, rectangle, minimum height = 3em, minimum width=0em]
\tikzstyle{pinstyle} = [pin edge={to-,thin,black}]
\tikzstyle{input} = [coordinate]
\tikzstyle{output} = [coordinate]

\newtheorem{thm}{Theorem}

\newtheorem{proposition}[thm]{Proposition}
\newtheorem{definition}[thm]{Definition}
\newtheorem{corollary}[thm]{Corollary}
\newtheorem{remark}[thm]{Remark}

\journal{Systems \& Control Letters}

\begin{document}

\begin{frontmatter}



\title{On a Canonical Distributed Controller in the Behavioral Framework\tnoteref{erc}}

\tnotetext[erc]{This work has received funding from the European Research Council (ERC), Advanced Research Grant SYSDYNET, under the European Union's Horizon 2020 research and innovation programme (Grant Agreement No. 694504).}

\author[label]{Tom R.V. Steentjes\corref{cor1}}
\ead{t.r.v.steentjes@tue.nl}
\cortext[cor1]{Corresponding author}
\author[label]{Mircea Lazar}
\ead{m.lazar@tue.nl}
\author[label]{Paul M.J. Van den Hof}
\ead{p.m.j.vandenhof@tue.nl}

            
\address[label]{Department of Electrical Engineering, Eindhoven University of Technology, P.O. Box 513, Eindhoven, 5600 MB, The Netherlands}            

\begin{abstract}
Control in a classical transfer function or state-space setting typically views a controller as a signal processor: sensor outputs are mapped to actuator inputs. In behavioral system theory, control is simply viewed as interconnection; the interconnection of a plant with a controller. In this paper we consider the problem of control of interconnected systems in a behavioral setting. The behavioral setting is especially fit for modelling interconnected systems, because it allows for the interconnection of subsystems without imposing inputs and outputs. We introduce a so-called \emph{canonical} distributed controller that implements a given interconnected behavior that is desired, provided that necessary and sufficient conditions hold true. The controller design can be performed in a decentralized manner, in the sense that a local controller only depends on the local system behavior. Regularity of interconnections is an important property in behavioral control that yields feedback interconnections. We provide conditions under which the interconnection of this distributed controller with the plant is regular. Furthermore, we show that the interconnections of subsystems of the canonical distributed controller are regular if and only if the interconnections of the plant and desired behavior are regular.
\end{abstract}




\end{frontmatter}


\section{Introduction}
When physical systems are interconnected, no distinction between inputs and outputs is made. Think for example of the interconnection of two RLC-circuits through their terminals or the interconnection of two mass-spring-damper systems. Typical transfer-function and input-state-output representations inherently impose an input-output partition of system variables. One of the main features of the \emph{behavioral} approach to system theory, is that it does not take an input-output structure as a starting point to describe systems: a mathematical model is simply the relation between system variables. In the case of dynamical systems, the set of all time trajectories that are compatible with the model is called the behavior. The behavioral approach has been advocated as a convenient starting point in several applications, among which in the context of interconnected systems~\citep{willems2007} and the context of control~\citep{willems97}.

In the context of interconnected systems, modelling can be performed through tearing (viewing the interconnected system as an interconnection of subsystems), zooming (modelling the subsystems), and linking (modelling the interconnections)~\citep{willems2007}. Interconnection of systems in a behavioral setting means variable sharing. When two masses are physically interconnected, the laws of motion for the first mass involve the position of the second mass and vice versa; the laws of motion of both masses together dictate the behavior of the interconnected system. Thinking of system  interconnections makes the modelling of interconnected systems remarkably simple. Partitioning variables into input and output variables is appropriate in signal processing, feedback control based on sensor outputs and other unilateral systems, but often unnecessary for physical system variables~\citep{willems2007}.

Feedback control based on sensor outputs to generate actuator inputs, where the controller is viewed as a signal processor~\citep{trentelman2011}, holds an important place in control theory. It has been argued that many practical control devices cannot be interpreted as feedback controllers, however, such as passive-vibration control systems, passive suspension systems or operational amplifiers~\citep{willems97}. Indeed, such control systems do not inherit a signal flow, but can be interpreted as an interconnection in a behavioral setting. Control by interconnection allows the control design to take place without distinguishing between control inputs and measured outputs, \emph{a priori}~\citep{willems97}, and can be performed for, e.g., stabilization~\citep{kuijper95}, $\mathscr{H}_\infty$ control~\citep{weiland97} and robust control~\citep{trentelman2011robust}.

Control by interconnection in a behavioral setting means restricting the behavior of the system that is to be controlled, by interconnecting it with a controller. By specifying a behavior that is desired for the controlled system, an important control problem is to determine the existence of a controller such that the controlled system's behavior is equal to the desired behavior. This is called the implementability problem~\citep{trentelman2011}. The \emph{canonical} controller plays a major role in the implementability problem: the canonical controller implements the desired behavior if and only if the desired behavior is implementable~\citep{vanderschaft2003}, \citep{julius2005}.

In this paper, we will consider \emph{distributed} control in a behavioral setting. In particular, we will consider distributed control of interconnected linear time-invariant systems. As a natural consequence of behavioral interconnections, we consider a distributed controller to be an interconnected system itself, i.e., we consider it to consist of subsystems that are interconnected without imposing signal flows between subsystems. Several types of interconnections become of interest in this problem: interconnections between subsystems of the to-be-controlled interconnected system (plant), interconnections between subsystems of the plant and subsystems of the distributed controller, and interconnections between subsystems of the distributed controller. Given a desired behavior for the controlled interconnected system that has the same interconnection structure as the plant, the considered distributed control problem is to determine the existence of a distributed controller that implements the desired behavior. We introduce a canonical distributed controller which implements the desired interconnected behavior under necessary and sufficient implementability conditions on the manifest plant and desired behavior. The distributed canonical controller has an attractive interconnection structure, in the sense that two of its subsystems are interconnected only if two subsystems of the plant or desired behavior are interconnected.

Distributed control with input-output partitioning and communication between subsystems of the distributed controller follows as an important special case of distributed control in a behavioral setting. An important question is: when can the canonical distributed controller be implemented with feedback interconnections? Following up on this question: When can the interconnections between controller subsystems be implemented as communication channels? The main concept in the solution to these problems is \emph{regularity} of the corresponding interconnections. We will analyze regularity of the canonical distributed controller. In particular, we show that the connections between subsystems of this distributed controller are regular if and only if connections between subsystems of the plant and desired behavior are regular.

\section{Preliminaries}
\subsubsection*{Behavioral notions} For the notions of systems in the behavioral setting, we will follow the
notation in~\citep{trentelman2011}. A dynamical system is defined as a triple $\Sigma=({T},{W},\mathfrak{B})$, where ${T}\subseteq \mathbb{R}$ is the {\color{black}time axis}, ${W}$ is the signal space and $\mathfrak{B}\subseteq{W}^{T}$ is the behavior. Consider two dynamical systems $\Sigma_1=({T},{W}_1\times{W}_3,\mathfrak{B}_1)$ and $\Sigma_2=({T},{W}_2\times{W}_3,\mathfrak{B}_2)$ with the same {\color{black}time axis}, and trajectories $(w_1,w_3)\in\mathfrak{B}_1$ and $(w_2,w_3)\in\mathfrak{B}_2$, respectively. The interconnection of $\Sigma_1$ and $\Sigma_2$ through $w_3$ yields the dynamical system
\begin{align*}
\Sigma_1\wedge_{w_3}\Sigma_2:=({T},{W}_1\times {W}_2\times {W}_3,\mathfrak{B}),
\end{align*}
with $\mathfrak{B}:=\{(w_1,w_2,w_3)\,|\, (w_1,w_3)\in\mathfrak{B}_1\text{ and } (w_2,w_3)\in\mathfrak{B}_2\}$.
The manifest behavior of $\Sigma_1$ with respect to $w_1$ is
\begin{align*}
(\mathfrak{B}_1)_{w_1}:=\{w_1:{T}\rightarrow{W}_1\,|\, \exists w_3 \text{ so that } (w_1,w_3)\in\mathfrak{B}\}.
\end{align*}
The set $\mathfrak{L}^\mathtt{w}$ denotes the set of all linear differential systems $\Sigma=(\mathbb{R},\mathbb{R}^\mathtt{w},\mathfrak{B})$, with $\mathtt{w}\in\mathbb{N}$ variables, where the behavior is
\begin{align*}
\mathfrak{B}:=\{w\in\mathfrak{C}^\infty(\mathbb{R},\mathbb{R}^\mathtt{w})\,|\, R(\frac{\dd }{\dd t})w=0\},
\end{align*}
with a polynomial matrix $R\in\mathbb{R}^{\mathtt{g}\times \mathtt{w}}[\xi]$, $\mathtt{g}\in\mathbb{N}_{>0}$, and $\mathfrak{C}^\infty(\mathbb{R},\mathbb{R}^\mathtt{w})$ denotes the set of infinitely often differentiable functions from $\mathbb{R}$ to $\mathbb{R}^\mathtt{w}$.

Consider a behavior $\mathfrak{B}\in\mathfrak{L}^\mathtt{w}$. The components of $w\in\mathfrak{B}$ allow for a component-wise partition\footnote{Up to re-ordering of the components in $w$.} such that $w=(u,y)$, with $u$ input and $y$ output. {\color{black}The partition $w=(u,y)$ is called an input-output partition if $u$ is free, i.e., for all $u$ there exists a $y$ so that $(u,y)\in\mathfrak{B}$, and $y$ does not contain any further free components, i.e., $u$ is \emph{maximally} free~\citep[Definition~3.3.1]{polderman98}, cf. \citep[Definition~2.9.2]{belur2003phd}}. The \emph{number} of components in the input and output, called the input and output cardinality, is invariant, i.e., independent of the input-output partition. Henceforth, $\mathtt{m}(\mathfrak{B})$ denotes the input cardinality and $\mathtt{p}(\mathfrak{B})$ denotes the output cardinality, which implies that $\mathtt{p}(\mathfrak{B})+\mathtt{m}(\mathfrak{B})=\mathtt{w}$. For a kernel representation $R\left(\frac{\dd }{\dd t}\right)w=0$ of $\mathfrak{B}$, the output cardinality is $\mathtt{p}(\mathfrak{B})=\operatorname{rank} R$.

\subsubsection*{Control by interconnection}\label{sec:controlbeh} A controlled interconnection is the interconnection of a plant $\Sigma_p=({T},{W}\times{C},\mathcal{P})$ and a controller $\Sigma_c=({T},{C},\mathcal{C})$, with the same {\color{black}time axis}, and trajectories $(w,c)\in\mathcal{P}$ and $c\in\mathcal{C}$, respectively. The plant has two types of variables: $w$ is the to-be-controlled variable and $c$ is the control variable. The controlled interconnection is thus $\mathcal{P}\wedge_c\mathcal{C}$. A general control problem can now be formulated as: Given the plant behavior $\mathcal{P}$ and a desired behavior $\mathcal{K}\subseteq {W}^{T}$, does there exist a controller $\mathcal{C}$ so that $\mathcal{K}=(\mathcal{P}\wedge_c\mathcal{C})_w$, i.e., is $\mathcal{K}$ implementable? {\color{black}The implementability problem has been extensively studied in \citep{trentelman2011, willemstrentelman2002}. Necessary and sufficient conditions for implementability are recalled in the following theorem.
\begin{thm}[\citep{willemstrentelman2002}] \label{thm:sandwich}
Let $\mathcal{P}\in\mathfrak{L}^{\mathtt{w}+\mathtt{c}}$ be a plant with $(\mathcal{P})_w\in\mathfrak{L}^{\mathtt{w}}$ its manifest behavior and $\mathcal{N}:=\{w\in\mathcal{P}\,|\, (w,0)\in\mathcal{P}\}$ its hidden behavior. Then $\mathcal{K}\in\mathfrak{L}^\mathtt{w}$ is implementable by a controller $\mathcal{C}\in\mathfrak{L}^\mathtt{c}$ if and only if
\begin{align*}
    \mathcal{N}\subseteq \mathcal{K}\subseteq(\mathcal{P})_w.
\end{align*}
\end{thm}}

\section{Control of interconnected systems}
\subsection{Plant interconnections}
For the design of a distributed controller, we consider {\color{black}$L$ systems (plants) $\Sigma_{p_i}=({T},{W}_i\times{S}_i\times{C}_i,\mathcal{P}_i)$, $i\in\mathbb{Z}_{[1:L]}:=\mathbb{Z}\cap [1, L]$}, having trajectories $(w_i,s_i,c_i)\in\mathcal{P}_i$, with $w_i$ the to-be-controlled variable, $s_i$ the inter-plant connection variable and $c_i$ the control variable. Partition the inter-plant connection variable $s_i$ into $s_{ij}$, the variable that behavior $\mathcal{P}_i$ shares with $\mathcal{P}_j$. {\color{black}The variable sharing is symmetric in the sense that if $\mathcal{P}_i$ shares variable $s_{ij}$ with $\mathcal{P}_j$, then $\mathcal{P}_j$ shares variable $s_{ji}$ with $\mathcal{P}_i$ and hence $s_{ij}=s_{ji}$.} The interconnection of $\mathcal{P}_i$ and $\mathcal{P}_j$ is given by
\begin{align*}
\mathcal{P}_i\wedge_{s_{ij}} \mathcal{P}_j=\{(w_i,w_j,s_{ij},c_i,c_j)\,|\, &(w_i,s_{ij},c_i)\in\mathcal{P}_i \quad\text{ and } \\
&(w_j,s_{ij},c_j)\in\mathcal{P}_j\}.
\end{align*}
We denote the straightforward generalization of the interconnection of $\mathcal{P}_i$, $i\in\mathbb{Z}_{[1:L]}$ as $\mathcal{P}_\mathcal{I}:=\wedge_{s_i,i\in\mathbb{Z}_{[1:L]}} \mathcal{P}_i$, such that
\begin{align*}
\mathcal{P}_\mathcal{I}=\{(w,s,c)\,|\, (w_i,s_i,c_i)\in\mathcal{P}_i\text{ for all } i\in\mathbb{Z}_{[1:L]}\}.
\end{align*}
Figure \ref{fig:interP} depicts an interconnection example of three behaviors $\mathcal{P}_1$, $\mathcal{P}_2$ and $\mathcal{P}_3$, through $s_{12}$ and $s_{23}$, i.e., $\mathcal{P}_1\wedge_{s_{12}}\mathcal{P}_2\wedge_{s_{23}}\mathcal{P}_3$. When we eliminate the interconnection variables $(s_i)_{i\in\mathbb{Z}_{[1:L]}}$ from the behavior of the interconnected system, $\mathcal{P}_\mathcal{I}$, we obtain the manifest behavior of $\mathcal{P}_\mathcal{I}$ with respect to $(w,c)$. This manifest behavior of the plant interconnection with respect to $(w,c)$ is $(\mathcal{P}_\mathcal{I})_{(w,c)}=(\wedge_{s_i,i\in\mathbb{Z}_{[1:L]}} \mathcal{P}_i)_{(w,c)}$, such that
\begin{align*}
(\mathcal{P}_\mathcal{I})_{(w,c)}&=\{(w,c)\,|\,\exists s_i\in\mathfrak{C}^\infty (\mathbb{R},\mathbb{R}^{\mathtt{s}_i}), i\in\mathbb{Z}_{[1:L]},\\
&\qquad\qquad  \text{ so that } (w_i,s_i,c_i)\in\mathcal{P}_i\text{ for all } i\in\mathbb{Z}_{[1:L]}\}.
\end{align*}

\begin{figure}[!t]
\centering
\subcaptionbox{System interconnection.\label{fig:interP}}{
\begin{tikzpicture}[auto,>=latex',node distance = 1.5em]
\matrix[ampersand replacement= \|, row sep=3em, column sep =1.5em] {
\node [var,draw = none](w1) {$w_1$};
\|
\node [block, fill = white, drop shadow](P1) {$\mathcal{P}_1$};
\|
\node [var,draw = none](c1) {$c_1$};
\\
\node [var,draw = none](w2) {$w_2$};
\|
\node [block, fill = white, drop shadow](P2) {$\mathcal{P}_2$};
\|
\node [var,draw = none](c2) {$c_2$};
\\
\node [var,draw = none](w3) {$w_3$};
\|
\node [block, fill = white, drop shadow](P3) {$\mathcal{P}_3$};
\|
\node [var,draw = none](c3) {$c_3$};
\\
};

\node [left = of P1.150] (w11){}; \node [left = of P1.180] (w12){}; \node [left = of P1.210] (w13){};

\node [left = of P2.150] (w21){}; \node [left = of P2.180] (w22){}; \node [left = of P2.210] (w23){};

\node [left = of P3.150] (w31){}; \node [left = of P3.180] (w32){}; \node [left = of P3.210] (w33){};

\node [right = of P1.30] (c11){}; \node [right = of P1.0] (c12){}; \node [right = of P1.330] (c13){};

\node [right = of P2.30] (c21){}; \node [right = of P2.0] (c22){}; \node [right = of P2.330] (c23){};

\node [right = of P3.30] (c31){}; \node [right = of P3.0] (c32){}; \node [right = of P3.330] (c33){};

\draw[-] (P1.270) --(P2.90); \draw[-] (P2.135)  -- node[]{$s_{12}$} (P1.225);
\draw[-] (P1.315) -- (P2.45);

\draw[-] (P2.270) -- (P3.90); \draw[-] (P3.135) -- node[]{$s_{23}$} (P2.225);
\draw[-] (P2.315) -- (P3.45);

\draw[-] (P1.150) -- (w11); \draw[-] (P1.180) -- (w12); \draw[-] (P1.210) -- (w13);
\draw[-] (P2.150) -- (w21); \draw[-] (P2.180) -- (w22); \draw[-] (P2.210) -- (w23);
\draw[-] (P3.150) -- (w31); \draw[-] (P3.180) -- (w32); \draw[-] (P3.210) -- (w33);

\draw[-] (P1.30) -- (c11); \draw[-] (P1.0) -- (c12); \draw[-] (P1.330) -- (c13);
\draw[-] (P2.30) -- (c21); \draw[-] (P2.0) -- (c22); \draw[-] (P2.330) -- (c23);
\draw[-] (P3.30) -- (c31); \draw[-] (P3.0) -- (c32); \draw[-] (P3.330) -- (c33);
\end{tikzpicture}
}
\subcaptionbox{Controller interconnection.\label{fig:interC}}{
\begin{tikzpicture}[auto,>=latex',node distance = 1.5em, color = black]
\matrix[ampersand replacement= \|, row sep=3em, column sep =1.5em] {
\node [var,draw = none](c1) {$c_1$};
\|
\node [blckblock, fill = white, drop shadow](C1) {$\mathcal{C}_1$};
\\
\node [var,draw = none](c2) {$c_2$};
\|
\node [blckblock, fill = white, drop shadow](C2) {$\mathcal{C}_2$};
\\
\node [var,draw = none](c3) {$c_3$};
\|
\node [blckblock, fill = white, drop shadow](C3) {$\mathcal{C}_3$};
\\
};

\node [left = of C1.150] (c11){}; \node [left = of C1.180] (c12){}; \node [left = of C1.210] (c13){};

\node [left = of C2.150] (c21){}; \node [left = of C2.180] (c22){}; \node [left = of C2.210] (c23){};

\node [left = of C3.150] (c31){}; \node [left = of C3.180] (c32){}; \node [left = of C3.210] (c33){};

\draw[-] (C1.270) --(C2.90); \draw[-] (C2.135)  -- (C1.225);
\draw[-] (C1.315) -- node[]{$p_{12}$} (C2.45);

\draw[-] (C2.270) -- (C3.90); \draw[-] (C3.135) -- (C2.225);
\draw[-] (C2.315) -- node[]{$p_{23}$} (C3.45);

\draw[-] (C1.150) -- (c11); \draw[-] (C1.180) -- (c12); \draw[-] (C1.210) -- (c13);
\draw[-] (C2.150) -- (c21); \draw[-] (C2.180) -- (c22); \draw[-] (C2.210) -- (c23);
\draw[-] (C3.150) -- (c31); \draw[-] (C3.180) -- (c32); \draw[-] (C3.210) -- (c33);

\end{tikzpicture}
}
\caption{Distributed control in the behavioral framework.}
\label{fig:distrscheme}
\end{figure}
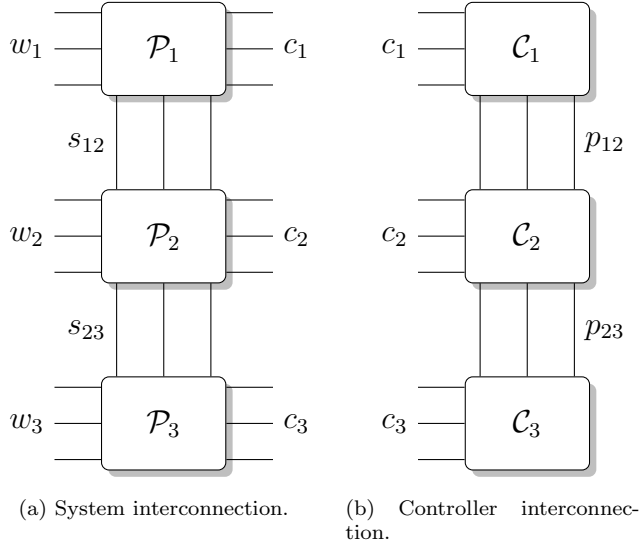

\subsection{Distributed control problem}
In the following, we will consider an interconnection of linear systems $\mathcal{P}_i\in\mathfrak{L}^{\mathtt{w}_i+\mathtt{s}_i+\mathtt{c}_i}$, $i\in\mathbb{Z}_{[1:L]}$. Given $\mathcal{K}_i\in\mathfrak{L}^{\mathtt{w}_i+\mathtt{k}_i}$, $i\in\mathbb{Z}_{[1:L]}$, let the desired behavior of the interconnected system be equal to the manifest behavior of the interconnection of $\mathcal{K}_i$ with respect to $w$, i.e.,
\begin{align*}
(\mathcal{K}_\mathcal{I})_w\!=\!(\wedge_{k_i,i\in\mathbb{Z}_{[1:L]}} &\mathcal{K}_i)_w\!=\!\{w\,|\, \exists k_i\!\in\!\mathfrak{C}^\infty (\mathbb{R},\mathbb{R}^{\mathtt{k}_i}), i\!\in\!\mathbb{Z}_{[1:L]},\\
&  \text{ so that } (w_i,k_i)\in\mathcal{K}_i, i\in\mathbb{Z}_{[1:L]}\}.
\end{align*}

Complementary to the interconnected plant, we are looking for another interconnected behavior, the controller, such that the interconnection of the plant with the controller yields the desired manifest behavior with respect to the to-be-controlled variable $w$, i.e., $(\mathcal{K}_\mathcal{I})_w$. The controller behavior is the interconnection of $\mathcal{C}_i\in\mathfrak{L}^{\mathtt{c}_i+\mathtt{p}_i}$, $i\in\mathbb{Z}_{[1:L]}$, through inter-controller connection variable $p_i$. The controller interconnection is distributed in the following sense: if $\mathcal{P}_i$ and $\mathcal{P}_j$ do not share a variable (they cannot be interconnected), then $\mathcal{C}_i$ and $\mathcal{C}_j$ do not share a variable, i.e., for each pair $(i,j)\in\mathbb{Z}_{[1:L]}^2$, it holds that $\mathtt{s}_{ij}=0$ $\Rightarrow$ $\mathtt{p}_{ij}=0$. In this way, the controller structure will reflect the plant structure and, hence, the structure of the ``closed-loop'' interconnection. This idea is exemplified in Figure~\ref{fig:interC} for the plant interconnection in Figure~\ref{fig:interP}. The chosen controller structure is a design choice that is natural in the sense  that the interconnection structure of the plant is respected. Therefore, this choice is commonly considered in the distributed control literature, cf.~\citep{andrea2003}, \citep{langbortetal2004}, \citep{camponogara2002}, \citep{rice2009}, \citep{chen2019}. Alternative distributed controller structures are, for example, hierarchical and multi-layer structures, which are designed according to multi-level or multi-resolution models~\citep{christofides2013} or through optimization~\citep{gusrialdi2012}.

Considering the `control by interconnection' problem described in Section~\ref{sec:controlbeh}, we can now analogously state the distributed control problem: Given the plant interconnection $\mathcal{P}_\mathcal{I}=\wedge_{s_i,i\in\mathbb{Z}_{[1:L]}} \mathcal{P}_i$ and a desired behavior defined by $\mathcal{K}_\mathcal{I}=\wedge_{k_i,i\in\mathbb{Z}_{[1:L]}} \mathcal{K}_i$, do there exist controllers $\mathcal{C}_i\in\mathfrak{L}^{\mathtt{c}_i+\mathtt{p}_i}$, $i\in\mathbb{Z}_{[1:L]}$, so that $(\mathcal{K}_\mathcal{I})_w=((\wedge_{s_i,i\in\mathbb{Z}_{[1:L]}} \mathcal{P}_i)\wedge_c (\wedge_{p_i, i\in\mathbb{Z}_{[1:L]}}\mathcal{C}_i))_w$? That is, does there exist a distributed controller such that the desired behavior is equal to the controlled interconnection? Figure~\ref{fig:distrcscheme} illustrates this controlled interconnection.
\begin{definition} \label{def:distrimpl}
Let $\mathcal{K}_i$, $i\in\mathbb{Z}_{[1:L]}$, be given and consider the desired interconnected system behavior $(\mathcal{K}_\mathcal{I})_w$. If there exists a distributed controller such that the controlled interconnected behavior equals the desired interconnected behavior, i.e., if there exist $\mathcal{C}_i$, $i\in\mathbb{Z}_{[1:L]}$, such that
\begin{gather}
(\wedge_{k_i,i\in\mathbb{Z}_{[1:L]}} \mathcal{K}_i)_w= \left([\wedge_{s_i,i\in\mathbb{Z}_{[1:L]}} \mathcal{P}_i]\wedge_c [\wedge_{p_i,i\in\mathbb{Z}_{[1:L]}} \mathcal{C}_i]\right)_w\nonumber\\
\Updownarrow  \label{eq:Knimplementable}\\
(\mathcal{K}_\mathcal{I})_w=\left(\mathcal{P}_\mathcal{I} \wedge_c [\wedge_{p_i,i\in\mathbb{Z}_{[1:L]}} \mathcal{C}_i]\right)_w\nonumber
\end{gather}
then $\mathcal{K}_I$ is called implementable via distributed control.

Consequently, a distributed controller with controller behaviors $\mathcal{C}_i$, $i\in\mathbb{Z}_{[1:L]}$, is said to implement $\mathcal{K}_\mathcal{I}$ if \eqref{eq:Knimplementable} holds.
\end{definition}

\begin{remark}
A natural question that comes to mind is: what prevents a desired behavior from being implementable? 
The necessary conditions of the controller implementability theorem for `centralized' control, recalled in Theorem~\ref{thm:sandwich}, reveal that there are two restrictions: (i) since control means that the behavior of the plant is restricted, the desired behavior must be a subset of the (manifest) behavior of the plant and (ii) since the hidden behavior of the plant (for $c=0$) should remain possible, the hidden behavior of the plant must be subset of the desired behavior~\citep{willemstrentelman2002}.
\end{remark}

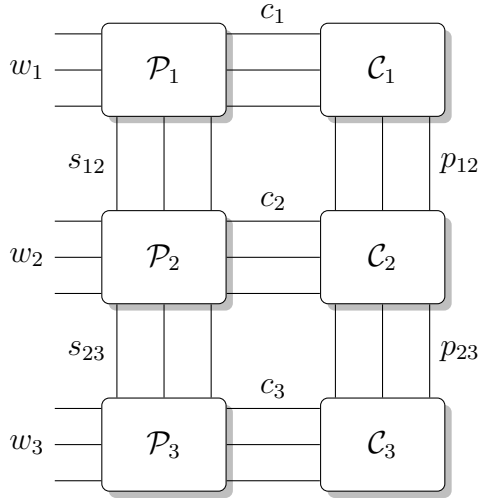
\begin{figure}[!htb]
\centering
\begin{tikzpicture}[auto,>=latex',node distance = 1.5em]
\matrix[ampersand replacement= \|, row sep=3em, column sep =1.5em](M1) {
\node [var,draw = none](w1) {$w_1$};
\|
\node [block, fill = white, drop shadow](P1) {$\mathcal{P}_1$};
\|\|
\node [blckblock, fill = white, drop shadow](C1) {$\mathcal{C}_1$};
\\
\node [var,draw = none](w2) {$w_2$};
\|
\node [block, fill = white, drop shadow](P2) {$\mathcal{P}_2$};
\|\|
\node [blckblock, fill = white, drop shadow](C2) {$\mathcal{C}_2$};
\\
\node [var,draw = none](w3) {$w_3$};
\|
\node [block, fill = white, drop shadow](P3) {$\mathcal{P}_3$};
\|\|
\node [blckblock, fill = white, drop shadow](C3) {$\mathcal{C}_3$};
\\
};

\node [left = of P1.150] (w11){}; \node [left = of P1.180] (w12){}; \node [left = of P1.210] (w13){};

\node [left = of P2.150] (w21){}; \node [left = of P2.180] (w22){}; \node [left = of P2.210] (w23){};

\node [left = of P3.150] (w31){}; \node [left = of P3.180] (w32){}; \node [left = of P3.210] (w33){};

\node [right = of P1.30] (c11){}; \node [right = of P1.0] (c12){}; \node [right = of P1.330] (c13){};

\node [right = of P2.30] (c21){}; \node [right = of P2.0] (c22){}; \node [right = of P2.330] (c23){};

\node [right = of P3.30] (c31){}; \node [right = of P3.0] (c32){}; \node [right = of P3.330] (c33){};

\draw[-] (P1.270) --(P2.90); \draw[-] (P2.135)  -- node[]{$s_{12}$} (P1.225);
\draw[-] (P1.315) -- (P2.45);

\draw[-] (P2.270) -- (P3.90); \draw[-] (P3.135) -- node[]{$s_{23}$} (P2.225);
\draw[-] (P2.315) -- (P3.45);

\draw[-] (P1.150) -- (w11); \draw[-] (P1.180) -- (w12); \draw[-] (P1.210) -- (w13);
\draw[-] (P2.150) -- (w21); \draw[-] (P2.180) -- (w22); \draw[-] (P2.210) -- (w23);
\draw[-] (P3.150) -- (w31); \draw[-] (P3.180) -- (w32); \draw[-] (P3.210) -- (w33);

{\color{black}
\draw[-] (C1.270) --(C2.90); \draw[-] (C2.135)  -- (C1.225);
\draw[-] (C1.315) -- node[]{$p_{12}$} (C2.45);

\draw[-] (C2.270) -- (C3.90); \draw[-] (C3.135) -- (C2.225);
\draw[-] (C2.315) -- node[]{$p_{23}$} (C3.45);

\draw[-] (P1.30) -- node[]{$c_1$} (C1.150); \draw[-] (C1.180) -- (P1.0); \draw[-] (C1.210) -- (P1.330);
\draw[-] (P2.30) -- node[]{$c_2$} (C2.150); \draw[-] (C2.180) -- (P2.0); \draw[-] (C2.210) -- (P2.330);
\draw[-] (P3.30) -- node[]{$c_3$} (C3.150); \draw[-] (C3.180) -- (P3.0); \draw[-] (C3.210) -- (P3.330);
}
\end{tikzpicture}
\caption{Controlled interconnection.}
\label{fig:distrcscheme}
\end{figure}

\section{Canonical distributed controller} \label{sec:ch5distrcan}
{\color{black}\subsection{Synthesis of a canonical distributed controller}}
Let $\mathcal{K}_i\in\mathfrak{L}^{\mathtt{w}_i+\mathtt{k}_i}$, $i\in\mathbb{Z}_{[1:L]}$, and consider its interconnected manifest behavior $(\mathcal{K}_\mathcal{I})_w\in\mathfrak{L}^{\mathtt{w}+\mathtt{k}}$. We define the controller
\begin{align} \label{eq:localcontr}
&\mathcal{C}_i^\text{can}:=(\mathcal{P}_i\wedge_{w_i}\mathcal{K}_i)_{(c_i,s_i,k_i)}\\
&=\!\{(c_i,s_i,k_i)\,|\, \exists w_i \text{ so that } {\color{black}(w_i,s_i,c_i)\in\mathcal{P}_i}\text{ and } (w_i,k_i)\in\mathcal{K}_i\},\nonumber
\end{align}
i.e., the manifest behavior with respect to $(c_i,s_i,k_i)$ of the interconnection of the local plant behavior $\mathcal{P}_i$ and desired behavior $\mathcal{K}_i$ through $w_i$. This interconnection is depicted in Figure \ref{fig:localc}. We call $\mathcal{C}_i^\text{can}$ a local canonical controller. By the elimination theorem~\citep[Theorem~6.2.6]{polderman98}, we have that $\mathcal{C}_i^\text{can}\in\mathfrak{L}^{\mathtt{c}_i+\mathtt{s}_i+\mathtt{k}_i}$.

Notice that by construction of the plant interconnection and interconnection defining the desired behavior, we can interconnect two canonical controllers $\mathcal{C}_i^\text{can}$ and $\mathcal{C}_j^\text{can}$ through the variables $(s_{ij},k_{ij})$, i.e., $\mathcal{C}_i^\text{can}\wedge_{(s_{ij},k_{ij})} \mathcal{C}_j^\text{can}$. In order to construct a distributed controller, we interconnect the local canonical controllers $\mathcal{C}_i^\text{can}$, $i\in\mathbb{Z}_{[1:L]}$, through $(s_i,k_i)$. The behavior of the interconnection of the local canonical controllers is
\begin{align} \label{eq:candistrcontr}
\mathcal{C}_\mathcal{I}^\text{can}=\wedge_{(s_i,k_i),i\in\mathbb{Z}_{[1:L]}}\mathcal{C}_i^\text{can},
\end{align}
which is called the canonical distributed controller.

{\color{black}\subsection{Implementability via distributed control}}
We will now provide conditions on the interconnected system and desired interconnected behavior under which the canonical distributed controller implements $\mathcal{K}_\mathcal{I}$. The hidden behavior of $\mathcal{P}_\mathcal{I}$ is defined as
\[\mathcal{N}(\mathcal{P}_\mathcal{I}):=\{w\,|\,(w,0)\in(\mathcal{P}_\mathcal{I})_{(w,c)}\}.\]
\begin{proposition} \label{prop:distrcan}
The canonical distributed controller $\mathcal{C}_\mathcal{I}^\text{can}$ implements the desired behavior $\mathcal{K}_\mathcal{I}\in\mathfrak{L}^{\mathtt{w}+\mathtt{k}}$ if and only if
\begin{align*}
\mathcal{N}(\mathcal{P}_\mathcal{I})\subseteq (\mathcal{K}_\mathcal{I})_w\subseteq (\mathcal{P}_\mathcal{I})_w.
\end{align*}
\begin{proof}
$(\Leftarrow)$ The sufficiency proof can be separated in two parts: (i) show that the distributed canonical controller satisfies ${\color{black}(\mathcal{C}_\mathcal{I}^\text{can})_c}=((\mathcal{P}_\mathcal{I})_{(w,c)}\wedge_w(\mathcal{K}_\mathcal{I})_w)_c$ and (ii) application of the implementability proof for the centralized canonical controller~\citep{willemstrentelman2002}, \citep{julius2005}. We will prove both parts (i) and (ii) for completeness.

We will first show that $(\wedge_{(s_i,k_i)}\mathcal{C}_i^\text{can})_c=\left((\wedge_{s_i}\mathcal{P}_i)_{(w,c)}\wedge_w (\wedge_{k_i} \mathcal{K}_i)_w\right)_c$, i.e., that ${\color{black}(\mathcal{C}_\mathcal{I}^\text{can})_c}=((\mathcal{P}_\mathcal{I})_{(w,c)}\wedge_w(\mathcal{K}_\mathcal{I})_w)_c$. The manifest behavior of $\wedge_{k_i}\mathcal{K}_i$ with respect to $w_i$ is
\begin{align*}
(\wedge_{k_i}\mathcal{K}_i)_w=\{(w_1,\dots,w_L)\,|\, \exists k_i,\ i\in\mathbb{Z}_{[1:L]},\text{ so that } (w_i,k_i)\in\mathcal{K}_i,\ i\in\mathbb{Z}_{[1:L]}\}
\end{align*}
and the manifest behavior of $\wedge_{s_i}\mathcal{P}_i$ with respect to $(w,c)$ is
\begin{align*}
(\wedge_{s_i}\mathcal{P}_i)_{(w,c)}=\{(w_1,\dots,w_L,c_1,\dots,c_L)\,|\, &\exists s_i,\ i\in\mathbb{Z}_{[1:L]},\text{ so that }\\
& (w_i,s_i,c_i)\in\mathcal{P}_i,\ i\in\mathbb{Z}_{[1:L]}\}.
\end{align*}
Hence, we have
\begin{align*}
\left((\wedge_{s_i}\mathcal{P}_i)_{(w,c)}\wedge_w (\wedge_{k_i} \mathcal{K}_i)_w\right)_c=\{(c_1,\dots,c_L)&\,|\, \exists (w_i,s_i,k_i),\ i\in\mathbb{Z}_{[1:L]},\text{ so that }\\
&(w_i,k_i)\in\mathcal{K}_i\text{ and } (w_i,s_i,c_i)\in\mathcal{P}_i\}.
\end{align*}
Furthermore, the manifest behavior of $\mathcal{C}_\mathcal{I}^\text{can}$ with respect to the control variable $c$ is
\begin{align*}
(\mathcal{C}_\mathcal{I}^\text{can})_c&=(\wedge_{(s_i,k_i),\ i\in\mathbb{Z}_{[1:L]}} \mathcal{C}_i^\text{can})_c\\
&=\{(c_1,\dots,c_L)\,|\, \exists (s_i,k_i),\ i\in\mathbb{Z}_{[1:L]},\text{ so that } (c_i,s_i,k_i)\in\mathcal{C}_i^\text{can}\},\\
&=\{(c_1,\dots,c_L)\,|\,\exists (w_i,s_i,k_i),\ i\in\mathbb{Z}_{[1:L]},\text{ so that }\\
 &\qquad (w_i,k_i)\in\mathcal{K}_i\text{ and } (w_i,s_i,c_i)\in\mathcal{P}_i\}.
\end{align*}
Hence, it follows that $({\color{black}\mathcal{C}_\mathcal{I}^\text{can})_c}=((\mathcal{P}_\mathcal{I})_{(w,c)}\wedge_w(\mathcal{K}_\mathcal{I})_w)_c$.

With this expression for the behavior of the canonical distributed controller, we find that the behavior of the interconnection of the manifest behavior of the canonical distributed controller and the manifest behavior of the plant is equal to
\begin{align*}
((\mathcal{P}_\mathcal{I})_{(w,c)}\wedge_c ({\color{black}(\mathcal{C}_\mathcal{I}^\text{can})_c})_w=((\mathcal{P}_\mathcal{I})_{(w,c)}\wedge_c ((\mathcal{P}_\mathcal{I})_{(w,c)}\wedge_w(\mathcal{K}_\mathcal{I})_w)_c)_w.
\end{align*}
We will now show that this behavior is in fact equal to $(\mathcal{K}_\mathcal{I})_w$. Consider minimal kernel representations for $(\mathcal{P}_\mathcal{I})_{(w,c)}$ and $(\mathcal{K}_\mathcal{I})_{w}$, respectively:
\begin{align*}
R\left(\tfrac{\dd }{\dd t}\right)w+M\left(\tfrac{\dd }{\dd t}\right)c=0,\quad K\left(\tfrac{\dd }{\dd t}\right)w=0.
\end{align*}
We therefore have that
\begin{align*}
\exists w\text { so that } \begin{bmatrix}
R\left(\frac{\dd }{\dd t}\right) & M\left(\frac{\dd }{\dd t}\right)\\ K\left(\frac{\dd }{\dd t}\right) & 0
\end{bmatrix}\begin{bmatrix}
w\\ c
\end{bmatrix}=0
\end{align*}
is a latent variable representation for $(\mathcal{C}_\mathcal{I}^\text{can})_c$. Since $\mathcal{N}(\mathcal{P}_\mathcal{I})\subseteq (\mathcal{K}_\mathcal{I})_w$ and $(\mathcal{K}_\mathcal{I})_w\subseteq (\mathcal{P}_\mathcal{I})_w$, there exists a polynomial matrix $F(\xi)$ so that $K(\xi)=F(\xi)R(\xi)$. Consider the {\color{black}polynomial} matrix $U(\xi):=\begin{bmatrix}
F(\xi) & -I\\ I & 0
\end{bmatrix}$. Post-multiplication of $U(\xi)$ with $\operatorname{col}(M(\xi), 0)$ yields $U(\xi)\begin{bmatrix}
M(\xi)\\ 0
\end{bmatrix}=\begin{bmatrix}
F(\xi)M(\xi)\\ M(\xi)
\end{bmatrix}$ and post-multiplication of $U(\xi)$ with $\operatorname{col}(-R(\xi),K(\xi))$ yields
\begin{align*}
U(\xi)\begin{bmatrix}
-R(\xi)\\ K(\xi)
\end{bmatrix}=\begin{bmatrix}
-F(\xi)R(\xi)+K(\xi)\\ -R(\xi)
\end{bmatrix}=\begin{bmatrix}
0\\ -R(\xi)
\end{bmatrix}.
\end{align*}
{\color{black}Since $U(\xi)$ is a unimodular matrix, the manifest behavior of $\mathcal{C}_\mathcal{I}^\text{can}$ with respect to $c$ consists of the $\mathfrak{C}^\infty$ solutions of $F\left(\frac{\dd }{\dd t}\right)M\left(\frac{\dd }{\dd t})\right)c=0$, by the elimination theorem~\citep[Theorem~6.2.6]{polderman98}.} We thus have ${\color{black}(\mathcal{C}_\mathcal{I}^\text{can})_c}=\{c\,|\, F\left(\frac{\dd }{\dd t}\right)M\left(\frac{\dd }{\dd t})\right)c=0\}$ so that
\begin{align*}
(\mathcal{P}_\mathcal{I})_{(w,c)}\wedge_c ((\mathcal{C}_\mathcal{I}^\text{can})_c&=\{(w,c)\,|\,\begin{bmatrix}
R\left(\frac{\dd }{\dd t}\right) & M\left(\frac{\dd }{\dd t}\right)\\ 0 & F\left(\frac{\dd }{\dd t}\right)M\left(\frac{\dd }{\dd t}\right)
\end{bmatrix}\begin{bmatrix}
w\\ c
\end{bmatrix}=0\}\\
&=\{(w,c)\,|\,\begin{bmatrix}
R\left(\frac{\dd }{\dd t}\right)\\ 0
\end{bmatrix}w=\begin{bmatrix}
-M\left(\frac{\dd }{\dd t}\right)\\ -F\left(\frac{\dd }{\dd t}\right)M\left(\frac{\dd }{\dd t}\right)
\end{bmatrix}c\}.
\end{align*}
Now, since
\begin{align*}
U(\xi)\begin{bmatrix}
R(\xi)\\ 0
\end{bmatrix}=\begin{bmatrix}
F(\xi)R(\xi)\\ R(\xi)
\end{bmatrix}\quad \text{ and }\quad U(\xi)\begin{bmatrix}
-M(\xi)\\-F(\xi)M(\xi)
\end{bmatrix}=\begin{bmatrix}
0\\ -M(\xi)
\end{bmatrix},
\end{align*}
we have $((\mathcal{P}_\mathcal{I})_{(w,c)}\wedge_c ((\mathcal{C}_\mathcal{I}^\text{can})_c)_w=\{w\,|\, FR\left(\tfrac{\dd }{\dd t}\right)w=0\}=\{w\,|\, K\left(\tfrac{\dd }{\dd t}\right)w=0\}=(\mathcal{K}_\mathcal{I})_w$.

$(\Rightarrow)$ Let the canonical distributed controller $\mathcal{C}_\mathcal{I}^\text{can}$ implement $\mathcal{K}_\mathcal{I}$. Then it holds that
\begin{align*}
    (\mathcal{K}_\mathcal{I})_w &=\left((\mathcal{P}_\mathcal{I})_{(w,c)} \wedge_c (\mathcal{C}_\mathcal{I}^\text{can})_c\right)_w\\
    &=\{w \,|\, \exists c\in (\mathcal{C}_\mathcal{I}^\text{can})_c \text{ so that } (w,c)\in(\mathcal{P}_\mathcal{I})_{(w,c)}\}.
\end{align*}
By definition, the manifest plant behavior with respect to the variable $w$ is given by
\begin{align*}
    (\mathcal{P}_\mathcal{I})_w=\{w\,|\, \exists c\in\mathfrak{C}^\infty(\mathbb{R},\mathbb{R}^\mathrm{c})\text{ so that } (w,c)\in(\mathcal{P}_\mathcal{I})_{(w,c)}\}.
\end{align*}
Hence, it follows that $(\mathcal{K}_\mathcal{I})_w\subseteq (\mathcal{P}_\mathcal{I})_w$. Further, the hidden behavior of $\mathcal{P}_\mathcal{I}$ is $\mathcal{N}(\mathcal{P}_\mathcal{I})=\{w\,|\,(w,0)\in(\mathcal{P}_\mathcal{I})_{(w,c)}\}$, which implies that $\mathcal{N}(\mathcal{P}_\mathcal{I})\subseteq (\mathcal{K}_\mathcal{I})_w$ and the proof is complete.
\end{proof}
\end{proposition}

\begin{figure}[!t]
\centering
\begin{tikzpicture}[auto,>=latex',node distance = 1.5em]
{\color{black}
\matrix[ampersand replacement= \|, row sep=3em, column sep =1.5em](M1) {
\node [var,draw = none](ci) {$c_i$};
\|
\node [blckblock, fill = white, drop shadow](Pi) {\reflectbox{$\mathcal{P}_i$}};
\|\|
\node [blckblock, fill = white, drop shadow](Ki) {$\mathcal{K}_i$};
\\
};
\node [left = of Pi.150] (ci1){}; \node [left = of Pi.180] (ci2){}; \node [left = of Pi.210] (ci3){};

\node [below = of Pi.270] (si1){}; \node [below = of Pi.225] (si2){}; \node [below = of Pi.315] (si3){};

\node [below = of Ki.270] (ki1){}; \node [below = of Ki.225] (ki2){}; \node [below = of Ki.315] (ki3){};

\draw[-] (Pi.150) -- (ci1); \draw[-] (Pi.180) -- (ci2); \draw[-] (Pi.210) -- (ci3);
\draw[-] (Pi.270) --(si1); \draw[-] (si2)  -- node[pos=.05]{$s_{i}$} (Pi.225);
\draw[-] (Pi.315) -- (si3);

\draw[-] (Ki.270) --(ki1); \draw[-] (ki2)  -- node[pos=.05]{$k_{i}$} (Ki.225);
\draw[-] (Ki.315) -- (ki3);

\draw[-] (Pi.30) -- node[]{$w_i$} (Ki.150); \draw[-] (Ki.180) -- (Pi.0); \draw[-] (Ki.210) -- (Pi.330);

\node[outerblock, fill opacity =0, fit=(Pi) (Ki)] (Ci) {};
\node [above =.5em of Ci.30] {$\mathcal{C}_i^\text{can}$};
}
\end{tikzpicture}
\caption{Local canonical controller. The mirrored plant notation emphasizes that the control and to-be-controlled variables of $\mathcal{P}_i$ are reversed inside the canonical controller.}
\label{fig:localc}
\end{figure}
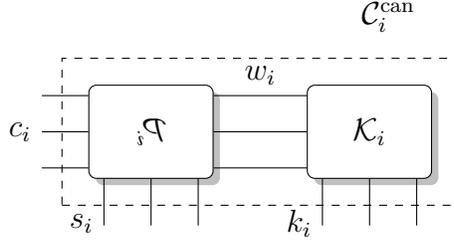

\begin{remark}
The manifest behavior of the controller~\eqref{eq:candistrcontr} with respect to the control variable $c$ is equal to the behavior of the ``central'' canonical controller for the desired interconnected behavior, cf.~\citep{julius2005}. Intuitively, this is sensible, see e.g. the controlled interconnection for the example with three subsystems in Figure~\ref{fig:distrpkscheme}. The controllers $\mathcal{C}_i^\text{can}$ are based on ``local'' behavior $\mathcal{P}_i$, while the central canonical controller is based on $(\mathcal{P}_\mathcal{I})_{(w,c)}$. {\color{black}From a synthesis point of view}, the control design is decentralized in the sense that only the subsystem $\mathcal{P}_i$ of the interconnected system is required to determine $\mathcal{C}_i^\text{can}$, once a desired interconnected behavior has been specified.
\end{remark}

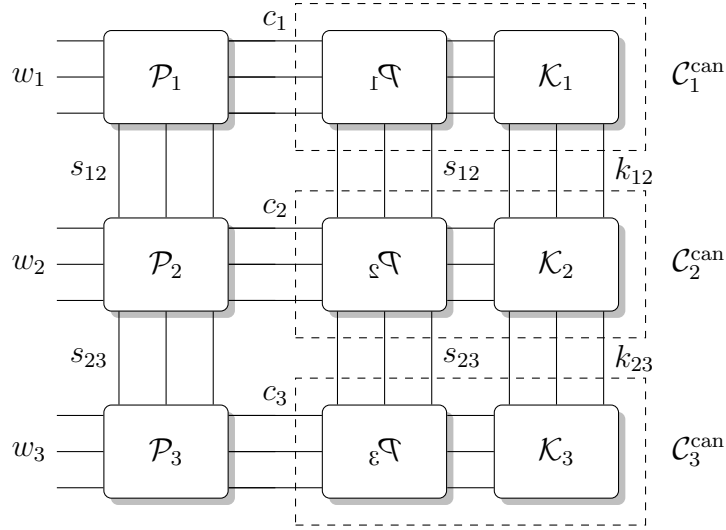
\begin{figure}[!t]
\centering
\begin{tikzpicture}[auto,>=latex',node distance = 1.5em]
\matrix[ampersand replacement= \|, row sep=3em, column sep =1.5em](M1) {
\node [var,draw = none](w1) {$w_1$};
\|
\node [block, fill = white, drop shadow](P1) {$\mathcal{P}_1$};
\|\|
\node [blckblock, fill = white, drop shadow](Pr1) {\reflectbox{$\mathcal{P}_1$}};
\|
\node [blckblock, fill = white, drop shadow](K1) {$\mathcal{K}_1$};
\\
\node [var,draw = none](w2) {$w_2$};
\|
\node [block, fill = white, drop shadow](P2) {$\mathcal{P}_2$};
\|\|
\node [blckblock, fill = white, drop shadow](Pr2) {\reflectbox{$\mathcal{P}_2$}};
\|
\node [blckblock, fill = white, drop shadow](K2) {$\mathcal{K}_2$};
\\
\node [var,draw = none](w3) {$w_3$};
\|
\node [block, fill = white, drop shadow](P3) {$\mathcal{P}_3$};
\|\|
\node [blckblock, fill = white, drop shadow](Pr3) {\reflectbox{$\mathcal{P}_3$}};
\|
\node [blckblock, fill = white, drop shadow](K3) {$\mathcal{K}_3$};
\\
};

\node [left = of P1.150] (w11){}; \node [left = of P1.180] (w12){}; \node [left = of P1.210] (w13){};

\node [left = of P2.150] (w21){}; \node [left = of P2.180] (w22){}; \node [left = of P2.210] (w23){};

\node [left = of P3.150] (w31){}; \node [left = of P3.180] (w32){}; \node [left = of P3.210] (w33){};

\node [right = of P1.30] (c11){}; \node [right = of P1.0] (c12){}; \node [right = of P1.330] (c13){};

\node [right = of P2.30] (c21){}; \node [right = of P2.0] (c22){}; \node [right = of P2.330] (c23){};

\node [right = of P3.30] (c31){}; \node [right = of P3.0] (c32){}; \node [right = of P3.330] (c33){};

\draw[-] (P1.270) --(P2.90); \draw[-] (P2.135)  -- node[]{$s_{12}$} (P1.225);
\draw[-] (P1.315) -- (P2.45);

\draw[-] (P2.270) -- (P3.90); \draw[-] (P3.135) -- node[]{$s_{23}$} (P2.225);
\draw[-] (P2.315) -- (P3.45);

\draw[-] (P1.150) -- (w11); \draw[-] (P1.180) -- (w12); \draw[-] (P1.210) -- (w13);
\draw[-] (P2.150) -- (w21); \draw[-] (P2.180) -- (w22); \draw[-] (P2.210) -- (w23);
\draw[-] (P3.150) -- (w31); \draw[-] (P3.180) -- (w32); \draw[-] (P3.210) -- (w33);

\draw[-] (P1.30) -- (c11); \draw[-] (P1.0) -- (c12); \draw[-] (P1.330) -- (c13);
\draw[-] (P2.30) -- (c21); \draw[-] (P2.0) -- (c22); \draw[-] (P2.330) -- (c23);
\draw[-] (P3.30) -- (c31); \draw[-] (P3.0) -- (c32); \draw[-] (P3.330) -- (c33);

\node [left = of Pr1.150] (c11){}; \node [left = of Pr1.180] (c12){}; \node [left = of Pr1.210] (c13){};

\node [left = of Pr2.150] (c21){}; \node [left = of Pr2.180] (c22){}; \node [left = of Pr2.210] (c23){};

\node [left = of Pr3.150] (c31){}; \node [left = of Pr3.180] (c32){}; \node [left = of Pr3.210] (c33){};

{\color{black}
\draw[-] (Pr1.270) --(Pr2.90); \draw[-] (Pr2.135)  -- (Pr1.225);
\draw[-] (Pr1.315) -- node[]{$s_{12}$} (Pr2.45);

\draw[-] (Pr2.270) -- (Pr3.90); \draw[-] (Pr3.135) -- (Pr2.225);
\draw[-] (Pr2.315) -- node[]{$s_{23}$} (Pr3.45);

\draw[-] (K1.270) --(K2.90); \draw[-] (K2.135)  -- (K1.225);
\draw[-] (K1.315) -- node[]{$k_{12}$} (K2.45);

\draw[-] (K2.270) -- (K3.90); \draw[-] (K3.135) -- (K2.225);
\draw[-] (K2.315) -- node[]{$k_{23}$} (K3.45);

\node [left = of K1.150] (wr11){}; \node [left = of K1.180] (wr12){}; \node [left = of K1.210] (wr13){};

\node [left = of K2.150] (wr21){}; \node [left = of K2.180] (wr22){}; \node [left = of K2.210] (wr23){};

\node [left = of K3.150] (wr31){}; \node [left = of K3.180] (wr32){}; \node [left = of K3.210] (wr33){};
\draw[-] (K1.150) -- (wr11); \draw[-] (K1.180) -- (wr12); \draw[-] (K1.210) -- (wr13);
\draw[-] (K2.150) -- (wr21); \draw[-] (K2.180) -- (wr22); \draw[-] (K2.210) -- (wr23);
\draw[-] (K3.150) -- (wr31); \draw[-] (K3.180) -- (wr32); \draw[-] (K3.210) -- (wr33);

\draw[-] (P1.30) -- node[]{$c_1$} (Pr1.150); \draw[-] (Pr1.180) -- (P1.0); \draw[-] (Pr1.210) -- (P1.330);
\draw[-] (P2.30) -- node[]{$c_2$} (Pr2.150); \draw[-] (Pr2.180) -- (P2.0); \draw[-] (Pr2.210) -- (P2.330);
\draw[-] (P3.30) -- node[]{$c_3$} (Pr3.150); \draw[-] (Pr3.180) -- (P3.0); \draw[-] (Pr3.210) -- (P3.330);

\node[outerblock, fill opacity =0, fit=(Pr1) (K1) ] (C1) {};
\node [right =.5em of C1.0] {$\mathcal{C}_1^\text{can}$};
\node[outerblock, fill opacity =0, fit=(Pr2) (K2) ] (C2) {};
\node [right =.5em of C2.0] {$\mathcal{C}_2^\text{can}$};
\node[outerblock, fill opacity =0, fit=(Pr3) (K3) ] (C3) {};
\node [right =.5em of C3.0] {$\mathcal{C}_3^\text{can}$};
}

\end{tikzpicture}
\caption{Controlled interconnection with local canonical controllers.}
\label{fig:distrpkscheme}
\end{figure}

\begin{remark}
The relevance of the developed canonical distributed controller reaches beyond the behavioral framework. An instance of the canonical distributed controller has been utilized for data-driven distributed control in~\citep{steentjes2020}, \citep{steentjes2021ecc} and \citep{steentjes2021sysid}, as a generalization of the ideal controller in model-reference control~\citep{campi2002}, \citep{bazanella2011}. The canonical distributed controller presented here is {\color{black}a} generalization of the ideal distributed controller in~\citep{steentjes2020} in the sense that it is representation free and does not distinguish between inputs and outputs of the interconnected system, both for the inter-plant connecting variables, as well as the control variables and inter-controller connecting variables. It is of future interest to determine how the canonical distributed controller can be further utilized in data-driven control, especially within the data-informativity framework~\citep{waarde2020tac}, \citep{vanwaarde2020noisytac} and e.g. data-driven distributed predictive control~\citep{kohler2022}, which currently relies on artificial state-space representations.
\end{remark}

{\color{black}
Implementability of a desired interconnected system behavior by the canonical distributed controller is clearly sufficient for implementability via distributed control as defined in Definition~\ref{def:distrimpl}. It is also necessary, by the following corollary. 
\begin{corollary} \label{cor:distrbeh}
 The desired interconnected system behavior $\mathcal{K}_\mathcal{I}$ is implementable via distributed control if and only if the canonical distributed controller $\mathcal{C}_\mathcal{I}^\text{can}$ implements $\mathcal{K}_\mathcal{I}$.
 \begin{proof}
 $(\Leftarrow)$ If the canonical distributed controller $\mathcal{C}_\mathcal{I}^\text{can}$ implements $\mathcal{K}_\mathcal{I}$, then \eqref{eq:Knimplementable} holds true for $\mathcal{C}_i=\mathcal{C}_i^\text{can}$, $i\in\mathbb{Z}_{[1:L]}$, and hence $\mathcal{K}_\mathcal{I}$ is implementable via distributed control.
 
$(\Rightarrow)$ Let $\mathcal{K}_\mathcal{I}$ be implementable via distributed control. Then there exist controllers $\mathcal{C}_i$, $i\in\mathbb{Z}_{[1:L]}$, such that $(\mathcal{K}_\mathcal{I})_w=\left(\mathcal{P}_\mathcal{I} \wedge_c [\wedge_{p_i,i\in\mathbb{Z}_{[1:L]}} \mathcal{C}_i]\right)_w$.
Following an analogous reasoning as in the necessity part of the proof for Proposition~\ref{prop:distrcan}, we have by definition that
\begin{align*}
    (\mathcal{K}_\mathcal{I})_w =\{w \,|\, \exists c\in (\mathcal{C}_\mathcal{I})_c \text{ so that } (w,c)\in(\mathcal{P}_\mathcal{I})_{(w,c)}\},
\end{align*}
which implies $(\mathcal{K}_\mathcal{I})_w\subseteq (\mathcal{P}_\mathcal{I})_w$. Indeed, by definition, the manifest plant behavior with respect to the variable $w$ is given by
\begin{align*}
    (\mathcal{P}_\mathcal{I})_w=\{w\,|\, \exists c\in\mathfrak{C}^\infty(\mathbb{R},\mathbb{R}^\mathrm{c})\text{ so that } (w,c)\in(\mathcal{P}_\mathcal{I})_{(w,c)}\}.
\end{align*}
Hence, it follows that $(\mathcal{K}_\mathcal{I})_w\subseteq (\mathcal{P}_\mathcal{I})_w$. Further, the hidden behavior of $\mathcal{P}_\mathcal{I}$ is $\mathcal{N}(\mathcal{P}_\mathcal{I})=\{w\,|\,(w,0)\in(\mathcal{P}_\mathcal{I})_{(w,c)}\}$, which implies that $\mathcal{N}(\mathcal{P}_\mathcal{I})\subseteq (\mathcal{K}_\mathcal{I})_w$. Therefore, the canonical distributed controller $\mathcal{C}_\mathcal{I}^\text{can}$ implements $\mathcal{K}_\mathcal{I}$ by Proposition~\ref{prop:distrcan}.
 \end{proof}
\end{corollary}

\begin{remark}
 The controller implementability conditions for a ``single'' system in Theorem~\ref{thm:sandwich} now follow as a special case from Proposition~\ref{prop:distrcan} and Corollary~\ref{cor:distrbeh}. Indeed, for a single system without inter-plant connection variables, i.e., $L=1$, $\mathtt{s}_1=0$ and $\mathtt{k}_1=0$, we have that $\mathcal{C}_\mathcal{I}^\text{can}=\mathcal{C}_1^\text{can}$ implements the desired behavior $\mathcal{K}_\mathcal{I}=\mathcal{K}_1$ if and only if $\mathcal{N}(\mathcal{P}_1)\subseteq \mathcal{K}_1\subseteq (\mathcal{P}_1)_w$, where we used $(\mathcal{K}_1)_w=\mathcal{K}_1$. Thus $\mathcal{K}_1$ is implementable by a controller $\mathcal{C}_1$ if and only if $\mathcal{N}(\mathcal{P}_1)\subseteq \mathcal{K}_1\subseteq (\mathcal{P}_1)_w$.
\end{remark}
}

\begin{figure}
    \centering
\begin{tikzpicture}[every node/.style={draw,outer sep=0pt,thick}]
\tikzstyle{spring}=[thick,decorate,decoration={zigzag,pre length=0.3cm,post length=0.3cm,segment length=5}]
\tikzstyle{damper}=[thick,decoration={markings,  
  mark connection node=dmp,
  mark=at position 0.5 with 
  {
    \node (dmp) [thick,inner sep=0pt,transform shape,rotate=-90,minimum width=15pt,minimum height=3pt,draw=none] {};
    \draw [thick] ($(dmp.north east)+(2pt,0)$) -- (dmp.south east) -- (dmp.south west) -- ($(dmp.north west)+(2pt,0)$);
    \draw [thick] ($(dmp.north)+(0,-5pt)$) -- ($(dmp.north)+(0,5pt)$);
  }
}, decorate]
\tikzstyle{ground}=[fill,pattern=north east lines,draw=none,minimum width=0.3cm,minimum height=0.75cm,pattern color=black]

\begin{scope}[xshift=7cm]
\node (M1) [minimum width=1cm, minimum height=1.25cm, fill=white] {$m_1$};
\node (M2) [minimum width=1cm, minimum height=1.25cm,right=of M1, xshift = 0.75cm,fill=white] {$m_2$};

\node (u1) [draw=none,right =of M1.45]{};
\node (u2) [draw=none,right =of M2.45]{};

\node (w1) [draw=none,right =of M1.90, yshift=0.5cm]{};
\node (w2) [draw=none,right =of M2.90, yshift=0.5cm]{};

\draw[->,very thick] (M1.45) -- node()[draw=none,yshift=0.25cm,xshift=0cm]{$d_1$}(u1);
\draw[->,very thick] (M2.45) -- node()[draw=none,yshift=0.25cm,xshift=0cm]{$d_2$}(u2);

\draw [thick] (M1.north) +(0cm,0.25cm) -- +(0cm,0.75cm);
\draw [->,thick] (M1.north) ++(0cm, 0.5cm) -- node()[draw=none,yshift=0.25cm]{$x_1$} +(1cm,0cm) (w1);

\draw [thick] (M2.north) +(0cm,0.25cm) -- +(0cm,0.75cm);
\draw [->,thick] (M2.north) ++(0cm, 0.5cm) -- node()[draw=none,yshift=0.25cm]{$x_2$} +(1cm,0cm) (w2);

\node (ground1) [ground,anchor=east,xshift=-1cm,minimum height=2.5cm,yshift=-1cm] at (M1.180) {};

\draw (ground1.north east) -- (ground1.south east);

\node (ground2) [ground,anchor=west,xshift=1cm,minimum height=2.5cm,yshift=-1cm] at (M2.0) {};

\draw (ground2.north west) -- (ground2.south west);

\draw [spring] (M1.0) -- ($(M2.north west)!(M1.0)!(M2.south west)$);
\draw [spring] (M1.180) -- ($(ground1.north east)!(M1.180)!(ground1.south east)$);
\draw [spring] (M2.0) -- ($(ground2.north west)!(M2.0)!(ground2.south west)$);

{\color{orange!85!black}
\node (M1c) [minimum width=1cm, minimum height=1cm, fill=white, below= of M1,yshift=0.75cm] {$m_1^c$};
\node (M2c) [minimum width=1cm, minimum height=1cm,right=of M1c, xshift = 0.75cm,fill=white] {$m_2^c$};

\draw [spring] (M1c.0) -- ($(M2c.north west)!(M1c.0)!(M2c.south west)$);

\draw [damper] (M1c.180) -- ($(ground1.north east)!(M1c.180)!(ground1.south east)$);
\draw [damper] (M2c.0) -- ($(ground2.north west)!(M2c.0)!(ground2.south west)$);

\draw [thick] (M1.south) -- (M1c.north);
\draw [thick] (M1.250) -- ($(M1c.north east)!(M1.250)!(M1c.north west)$);
\draw [thick] (M1.290) -- ($(M1c.north east)!(M1.290)!(M1c.north west)$);

\draw [thick] (M2.south) -- (M2c.north);
\draw [thick] (M2.250) -- ($(M2c.north east)!(M2.250)!(M2c.north west)$);
\draw [thick] (M2.290) -- ($(M2c.north east)!(M2.290)!(M2c.north west)$);
}

\end{scope}
\end{tikzpicture}
\caption{\color{black}Interconnected mass-spring system example (black) and a physical interpretation of the canonical distributed controller $\mathcal{C}_\mathcal{I}^\text{can}$ (orange).}
    \label{fig:physexample}
\end{figure}
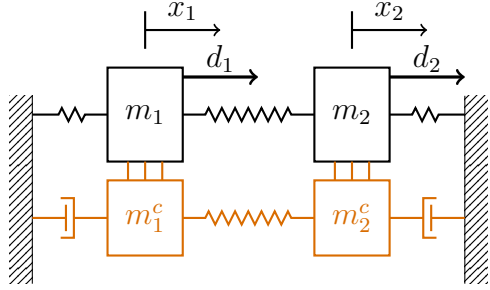

{\color{black}\subsection{Example: interconnected mass-spring system}
Consider an interconnected mass-spring system, with unity masses and spring constants, described by
\begin{align*}
    \frac{\dd^2}{\dd t^2}x_1+(x_1-x_2)+x_1&=f_1+d_1,\\
    \frac{\dd^2}{\dd t^2}x_2+(x_2-x_1)+x_2&=f_2+d_2,
\end{align*}
where $x_i$, $i=1,2,$ are the mass positions and $f_i$, $d_i$ are external forces acting on the masses. We describe the interconnected system by behaviors $\mathcal{P}_1$, $\mathcal{P}_2$ with variables $w_i:=\operatorname{col}(x_i,d_i)$, $c_i:=\operatorname{col}(x_i,f_i)$ and $s=\operatorname{col}(x_1,x_2)$, admitting kernel representations $P_i(\frac{\dd }{\dd t})\operatorname{col}(w_i,s,c_i)=0$, with
\begin{align*}
    P_1(\xi)&:=\begin{bmatrix}
        \xi^2+2 & -1 & 0 & -1 & 0 & -1\\ -1 & 0 & 1 & 0 & 0 & 0\\ -1 & 0 & 0 & 0 & 1 & 0
    \end{bmatrix}\quad \text{ and }\\
    P_2(\xi)&:=\begin{bmatrix}
        \xi^2+2 & -1 & -1 & 0 & 0 & -1\\ -1 & 0 & 0 & 1 & 0 & 0\\ -1 & 0 & 0 & 0 & 1 & 0
    \end{bmatrix}.
\end{align*}

For the desired interconnected system, we wish that the controlled behavior represents a system with increased mass, as well as damping between masses and ground, and a higher stiffness between the masses. Specifically, the desired behavior is described by the differential equations
\begin{align*}
    2\frac{\dd^2}{\dd t^2}x_1+2(x_1-x_2)+\frac{\dd }{\dd t}x_1+x_1&=d_1,\\
    2\frac{\dd^2}{\dd t^2}x_2+2(x_2-x_1)+\frac{\dd }{\dd t}x_2+x_2&=d_2.
\end{align*}
The desired interconnected behavior is the interconnection of $\mathcal{K}_1$, $\mathcal{K}_2$, through $k:=\operatorname{col}(x_1,x_2)$, admitting kernel representations $K_i(\frac{\dd }{\dd t})\operatorname{col}(w_i,k)=0$, with
\begin{align*}
    K_1(\xi)&:=\begin{bmatrix}
        2\xi^2+\xi+3 & -1 & 0 & -2\\ -1 & 0 & 1 & 0
    \end{bmatrix}\quad \text{ and }\\
    K_2(\xi)&:=\begin{bmatrix}
        2\xi^2+\xi+3 & -1 & -2 & 0\\ -1 & 0 & 0 & 1
    \end{bmatrix}.
\end{align*}
The desired interconnected system satisfies the condition in Proposition~\ref{prop:distrcan}, thus the canonical distributed controller implements the desired behavior. The local canonical controllers in \eqref{eq:localcontr} admit the kernel representations $C_i^\text{can}(\frac{\dd}{\dd t})\operatorname{col}(c_i,s,k)=0$, with
\begin{align*}
    C_1^\text{can}(\xi)&:=\begin{bmatrix}
        \xi^2+\xi+1 & 1 & 0 & 1 & 0 & -2\\
        0 & 0 & -1 & 0 & 1 & 0\\
        1 & 0 & -1 & 0 & 0 & 0
    \end{bmatrix}\quad \text{ and }\\
    C_2^\text{can}(\xi)&:=\begin{bmatrix}
        \xi^2+\xi+1 & 1 & 1 & 0 & -2 & 0\\
        0 & 0 & 0 & -1 & 0 & 1\\
        1 & 0 & 0 & -1 & 0 & 0
    \end{bmatrix}.
\end{align*}
From the interconnection $\mathcal{C}_\mathcal{I}^\text{can}= \mathcal{C}_1^\text{can}\wedge_{(s,k)}\mathcal{C}_2^\text{can}$, we observe that $k=s$, $\begin{bmatrix}
    1 & 0
\end{bmatrix}c_1=\begin{bmatrix}
    1 & 0
\end{bmatrix}k$, $\begin{bmatrix}
    1 & 0
\end{bmatrix}c_2=\begin{bmatrix}
    0 & 1
\end{bmatrix}k$ and 
\begin{align}
    \frac{\dd^2}{\dd t^2}[c_1]_1+\frac{\dd}{\dd t}[c_1]_1+[c_1]_1+[c_1]_2+\begin{bmatrix}
        0 & -1
    \end{bmatrix}k&=0, \label{xmpl:1}\\ 
    \frac{\dd^2}{\dd t^2}[c_2]_1+\frac{\dd}{\dd t}[c_2]_1+[c_2]_1+[c_2]_2+\begin{bmatrix}
        -1 & 0
    \end{bmatrix}k&=0. \label{xmpl:2}
\end{align}
The canonical distributed controller behavior thus has a clear physical interpretation. Indeed, the distributed controller can be represented by a mass-spring-damper system with unity masses, unity damping coefficients and a spring between the masses with a unity stiffness. The physical interpretation of the canonical distributed controller and its interconnection with the plants is visualized in Figure~\ref{fig:physexample}.

A remark on the necessity of the interconnection between the local canonical controllers: In the case that the stiffness interconnecting the masses is required to be the same for the plant and desired interconnected system, then equations \eqref{xmpl:1} and \eqref{xmpl:2} become decoupled in the sense that there is no dependency on $k$, i.e., the last term in \eqref{xmpl:1} and \eqref{xmpl:2} becomes $\begin{bmatrix}
    0 & 0
\end{bmatrix}k$. Therefore, the orange spring in the physical interpretation in Figure~\ref{fig:physexample} is not present and the local canonical controllers are thus not necessarily coupled.
}

\section{Regularity of the canonical distributed controller}
An important type of system interconnections is a \emph{regular interconnection}, introduced by Willems in~\citep{willems97}. Formally, a regular interconnection of two systems is defined as follows.
\begin{definition} \label{def:ch5reg}
Consider two behaviors $\mathfrak{B}_1\in\mathfrak{L}^{\mathtt{w}_1+\mathtt{w}_2}$ and $\mathfrak{B}_2\in\mathfrak{L}^{\mathtt{w}_2+\mathtt{w}_3}$. The interconnection of $\mathfrak{B}_1$ and $\mathfrak{B}_2$ is said to be regular if $\mathtt{p}(\mathfrak{B}_1\wedge_{w_2} \mathfrak{B}_2)=\mathtt{p}(\mathfrak{B}_1)+\mathtt{p}(\mathfrak{B}_2)$, where $\mathfrak{B}_1\wedge_{w_2} \mathfrak{B}_2=\{(w_1,w_2,w_3)\,|\, (w_1,w_2)\in\mathfrak{B}_1 \text{ and } (w_2,w_3)\in\mathfrak{B}_2\}$.
\end{definition}
Regularity of the interconnection of two systems has multiple interpretations. First, regularity means in a sense that the equations describing the dynamics of $\mathfrak{B}_1$ and $\mathfrak{B}_2$ are independent of each other~\citep{willems2003cdc}. For the second interpretation, consider a plant $\mathcal{P}\in\mathfrak{L}^{\mathtt{w}+\mathtt{c}}$, a controller $\mathcal{C}\in\mathfrak{L}^\mathtt{c}$ and their interconnection $\mathcal{K}:=\{(w,c)\in\mathcal{P}\,|\, c\in\mathcal{C}\}$. According to Definition~\ref{def:ch5reg}, the plant-controller interconnection is regular if
\begin{align*}
\mathtt{p}(\mathcal{K})=\mathtt{p}(\mathcal{P})+\mathtt{p}(\mathcal{C}).
\end{align*}
This interconnection is regular if and only if the controller $\mathcal{C}$ can be realized as a transfer function from an output variable to an input variable of $\mathcal{P}$ for an input/output partitioning of the control variable $c$~\citep{willems2003cdc}. From a control-point-of-view, regularity of the plant-controller interconnection therefore means that the controller acts as a feedback controller, i.e., it can process sensor outputs to actuator inputs. {\color{black}Notice that this typical assumption in classical and modern control theory is not assumed \emph{a priori} in control in a behavioral setting~\citep{willems97}, \citep{willems2003cdc}, \citep{julius2005}. A special type of regular interconnections is a regular feedback interconnection for which, in addition to being regular, the sum of the McMillan degrees of $\mathcal{P}$ and $\mathcal{C}$ is equal to the McMillan degree of the interconnection of $\mathcal{P}$ and $\mathcal{C}$~\citep[Definition~2.5]{vinjamoorbelur2010}. In practice, regular feedback interconnections avoid `impulsive' behavior when two systems are interconnected, such as sparks in electrical switching and jerky behavior in mechanical interconnections~\citep{vinjamoorbelur2010}.}

Let us now consider regularity of the interconnections related to the canonical distributed controller, which was introduced in Section~\ref{sec:ch5distrcan}. There are two types of interconnections that are of interest: (i) the interconnection between the canonical distributed  controller {\color{black}$\mathcal{C}_\mathcal{I}^\text{can}$} and the interconnected system $\mathcal{P}_\mathcal{I}$, i.e., the plant-controller interconnection and (ii) the interconnection between  $\mathcal{C}_i^\text{can}$ and $\mathcal{C}_j^\text{can}$, $(i,j)\in\mathbb{Z}_{[1:L]}^2$ and $i\neq j$, i.e., the interconnection of local controllers. The interpretation of regularity of the plant-controller interconnection has been considered in the previous paragraph. Regularity of the interconnection of local canonical controllers can be interpreted as follows. If the interconnection between controllers is regular, then the interconnection variable $p_{ij}$ can always be partitioned to achieve a regular feedback interconnection, i.e., such that the transfers from inputs in the partitioning to outputs are proper. Regularity of the interconnection between local controllers thus means that the controllers can communicate by processing received signals (input) into sent signals (output).

\subsection{Regularity of the plant-controller interconnection}
Regularity of the interconnection of the interconnected system behavior $\mathcal{P}_\mathcal{I}$ and a distributed controller $\mathcal{C}_\mathcal{I}$ follows from the regularity of the behaviors with the interconnection variables $(s_1,\dots,s_L)$ and $(p_1,\dots, p_L)$ eliminated, i.e., from $(\mathcal{P}_\mathcal{I})_{(w,c)}$ and $(\mathcal{C}_\mathcal{I})_c$. By definition, the interconnection of $(\mathcal{P}_\mathcal{I})_{(w,c)}$ and $(\mathcal{C}_\mathcal{I})_c$ is regular if
\begin{align} \label{eq:ch5cdregular}
\mathtt{p}((\mathcal{P}_\mathcal{I})_{(w,c)})+\mathtt{p}((\mathcal{C}_\mathcal{I})_{c})=\mathtt{p}((\mathcal{P}_\mathcal{I})_{(w,c)}\wedge_c (\mathcal{C}_\mathcal{I})_c).
\end{align}
If~\eqref{eq:ch5cdregular} holds, then the distributed controller is called regular with respect to the variable $c$. A sufficient condition for regularity with respect to the variable $c$  of all distributed controllers that implement $\mathcal{K}_\mathcal{I}$ follows from~\citep[Theorem~12]{julius2005}.

\begin{proposition}
Let $\mathcal{P}_i\in\mathfrak{L}^{\mathtt{w}_i+\mathtt{s}_i+\mathtt{c}_i}$ and $\mathcal{C}_i\in\mathfrak{L}^{\mathtt{c}_i+\mathtt{p}_i}$, $i\in\mathbb{Z}_{[1:L]}$, and consider the interconnected system $\mathcal{P}_\mathcal{I}=\wedge_{s_i,i\in\mathbb{Z}_{[1:L]}} \mathcal{P}_i$ and distributed controller $\mathcal{C}_\mathcal{I}=\wedge_{p_i,i\in\mathbb{Z}_{[1:L]}} \mathcal{C}_i$. Let $(\mathcal{K}_\mathcal{I})_w$ be the desired behavior, with $\mathcal{K}_\mathcal{I}=\wedge_{k_i,i\in\mathbb{Z}_{[1:L]}} \mathcal{K}_i$, where $\mathcal{K}_i\in\mathfrak{L}^{\mathtt{w}_i+\mathtt{k}_i}$, $i\in\mathbb{Z}_{[1:L]}$.

Every distributed controller $\mathcal{C}_\mathcal{I}$ that implements $\mathcal{K}_\mathcal{I}$, i.e., \eqref{eq:Knimplementable} holds, is regular with respect to the variable $c$ if $(\mathcal{P}_\mathcal{I})_c=\mathfrak{C}^\infty(\mathbb{R},\mathbb{R}^\mathtt{c})$, where $(\mathcal{P}_\mathcal{I})_c$ is the manifest behavior of the interconnected system with respect to the variable $c$, i.e.,
\begin{align*}
(\mathcal{P}_\mathcal{I})_c=\{c\,|\, \exists (w,s) \text{ so that } (w,s,c)\in\mathcal{P}_\mathcal{I}\}.
\end{align*}
\end{proposition}
\begin{proof}
First, notice that $\mathcal{P}_\mathcal{I}=\wedge_{s_i,i\in\mathbb{Z}_{[1:L]}} \mathcal{P}_i\in\mathfrak{L}^{\mathtt{w}+\mathtt{s}+\mathtt{c}}$ and that $(\mathcal{P}_\mathcal{I})_{(w,c)}\in\mathfrak{L}^{\mathtt{w}+\mathtt{c}}$. Hence, there exists a minimal kernel representation for $(\mathcal{P}_\mathcal{I})_{(w,c)}$:
\begin{align*}
R\left(\frac{\dd }{\dd t}\right)w+M\left(\frac{\dd }{\dd t}\right) c=0.
\end{align*}
Assume that $(\mathcal{P}_\mathcal{I})_c=\mathfrak{C}^\infty(\mathbb{R},\mathbb{R}^\mathtt{c})$. Then $R$ has full row rank. Now, take any distributed controller $\mathcal{C}_\mathcal{I}=\wedge_{p_i,i\in\mathbb{Z}_{[1:L]}} \mathcal{C}_i\in\mathfrak{L}^{\mathtt{c}+\mathtt{p}}$ that implements $\mathcal{K}_\mathcal{I}$. The manifest behavior of $\mathcal{C}_\mathcal{I}$ with respect to the variable $c$, i.e., $(\mathcal{C}_\mathcal{I})_c$, satisfies $(\mathcal{C}_\mathcal{I})_c\in\mathfrak{L}^{\mathtt{c}}$ and therefore has a minimal kernel representation $C\left(\frac{\dd }{\dd t}\right) c=0$. Since $R$ has full row rank, we find that 
\begin{align*}
\begin{bmatrix}
R\left(\frac{\dd }{\dd t}\right) & M\left(\frac{\dd }{\dd t}\right)\\ 0 & C\left(\frac{\dd }{\dd t}\right)
\end{bmatrix}\begin{bmatrix}
w\\ c
\end{bmatrix}=0
\end{align*}
is a minimal kernel representation of $(\mathcal{K}_\mathcal{I})_{(w,c)}$. We find that
\begin{align*}
\mathtt{p}((\mathcal{K}_\mathcal{I})_{(w,c)})=\operatorname{rank} R+\operatorname{rank} C=\mathtt{p}((\mathcal{P}_\mathcal{I})_{(w,c)})+\mathtt{p}((\mathcal{C}_\mathcal{I})_{c}),
\end{align*}
which was to be proven.
\end{proof}

\begin{corollary}
Consider an interconnected system $\mathcal{P}_\mathcal{I}=\wedge_{s_i,i\in\mathbb{Z}_{[1:L]}} \mathcal{P}_i$, $\mathcal{P}_i\in\mathfrak{L}^{\mathtt{w}_i+\mathtt{s}_i+\mathtt{c}_i}$, and the desired behavior $(\mathcal{K}_\mathcal{I})_w$, with $\mathcal{K}_\mathcal{I}=\wedge_{k_i,i\in\mathbb{Z}_{[1:L]}} \mathcal{K}_i$, $\mathcal{K}_i\in\mathfrak{L}^{\mathtt{w}_i+\mathtt{k}_i}$, $i\in\mathbb{Z}_{[1:L]}$. Assume that
\begin{align*}
\mathcal{N}(\mathcal{P}_\mathcal{I})\subseteq (\mathcal{K}_\mathcal{I})_w\subseteq (\mathcal{P}_\mathcal{I})_w.
\end{align*}
If $(\mathcal{P}_\mathcal{I})_c=\mathfrak{C}^\infty(\mathbb{R},\mathbb{R}^\mathtt{c})$, then the canonical distributed controller implements $\mathcal{K}_\mathcal{I}$ and is regular with respect to the variable $c$.
\end{corollary}

\subsection{Regularity of the interconnection of local canonical controllers}
Let us now consider the regularity of the interconnection of local canonical controllers, i.e., the regularity of $\mathcal{C}_i^\text{can}\wedge_{(s_{ij},k_{ij})}\mathcal{C}_j^\text{can}$, $(i,j)\in\mathbb{Z}_{[1:L]}^2$ and $i\neq j$. Without loss of generality, we will consider that $L=2$ in this subsection. The interconnection of $\mathcal{C}_1^\text{can}$ and $\mathcal{C}_2^\text{can}$ is regular if
\begin{align*}
\mathtt{p}(\mathcal{C}_1^\text{can}\wedge_{(s,k)}\mathcal{C}_2^\text{can})=\mathtt{p}(\mathcal{C}_1^\text{can})+\mathtt{p}(\mathcal{C}_2^\text{can}).
\end{align*}

The behaviors $\mathcal{P}_i\in\mathfrak{L}^{\color{black}\mathtt{w}_i+\mathtt{s}+\mathtt{c}_i}$, $i=1,2$, admit kernel representations
\begin{align} \label{eq:ch5kerpi}
R_i\left(\frac{\dd}{\dd t}\right)w_i+S_i\left(\frac{\dd}{\dd t}\right)s+M_i\left(\frac{\dd}{\dd t}\right)c_i=0,\quad i=1,2.
\end{align}
Similarly, the behaviors $\mathcal{K}_i\in\mathfrak{L}^{\color{black}\mathtt{w}_i+\mathtt{k}}$, $i=1,2$, admit kernel representations
\begin{align} \label{eq:ch5kerki}
W_i\left(\frac{\dd}{\dd t}\right)w_i+K_i\left(\frac{\dd}{\dd t}\right)k=0,\quad i=1,2.
\end{align}
Define the partitioned matrix
\begin{align} \label{eq:ch5LNmat}
\left[\begin{array}{cccc|cc}
M_1 & 0 & S_1 & 0 & R_1 & 0\\ 0 & 0 & 0 & K_1 & W_1 & 0\\ \hline 0 & M_2 & S_2 & 0 & 0 & R_2\\ 0 & 0 & 0 & K_2 & 0 & W_2
\end{array}\right]=:\left[\begin{array}{c|c}
L_1 & N_1\\ \hline L_2 & N_2
\end{array}\right].
\end{align}

\begin{proposition} \label{prop:ch5regcican}
Consider the behaviors $\mathcal{P}_i\in\mathfrak{L}^{\color{black}\mathtt{w}_i+\mathtt{s}+\mathtt{c}_i}$ and $\mathcal{K}_i\in\mathfrak{L}^{\color{black}\mathtt{w}_i+\mathtt{k}}$, $i=1,2$, and the kernel representations~\eqref{eq:ch5kerpi} and \eqref{eq:ch5kerki}, respectively. The interconnection of $\mathcal{C}_1^\text{can}$ and $\mathcal{C}_2^\text{can}$ is regular if and only if
\begin{align} \label{eq:ch5propln}
\operatorname{rank}\begin{bmatrix}
L_1 & N_1
\end{bmatrix}+\operatorname{rank}\begin{bmatrix}
L_2 & N_2
\end{bmatrix}=\operatorname{rank}\begin{bmatrix}
L_1 & N_1\\ L_2 & N_2
\end{bmatrix}.
\end{align}
\end{proposition}
\begin{proof}
By~\eqref{eq:ch5kerpi} and~\eqref{eq:ch5kerki}, the local canonical controller behavior is represented by the latent variable representation
\begin{align*}
&\mathcal{C}_i^\text{can}=\{(c_i,s,k)\,|\, \exists w_i \text{ so that }\\
&\begin{bmatrix}
R_i\!\left(\!\frac{\dd}{\dd t}\right) & S_i\!\left(\!\frac{\dd}{\dd t}\right) & M_i\!\left(\!\frac{\dd}{\dd t}\right) & 0\\ W_i\!\left(\!\frac{\dd}{\dd t}\right) & 0 & 0 & K_i\!\left(\!\frac{\dd}{\dd t}\right)
\end{bmatrix}\!\!\begin{bmatrix}
w_i\\ c_i\\ s\\ k
\end{bmatrix}=0\}.
\end{align*}
Hence, the interconnection of $\mathcal{C}_1^\text{can}$ and $\mathcal{C}_2^\text{can}$ is
\begin{align*}
&\mathcal{C}_1^\text{can}\wedge_{(s,k)} \mathcal{C}_2^\text{can}\\
&=\{(c_1,c_2,s,k)\,|\, (c_1,s,k)\in \mathcal{C}_1^\text{can}\text{ and } (c_2,s,k)\in \mathcal{C}_2^\text{can}\}\\
&=\{(c,s,k)\,|\,\exists w \text{ so that } \begin{bmatrix}
L_1\left(\frac{\dd}{\dd t}\right) & N_1\left(\frac{\dd}{\dd t}\right)\\ L_2\left(\frac{\dd}{\dd t}\right) & N_2\left(\frac{\dd}{\dd t}\right)
\end{bmatrix}\left[\begin{array}{c}
c\\s\\k\\\hline w
\end{array}\right]=0\},
\end{align*}
which is a latent variable representation for the canonical distributed controller (with latent variable $(w_1,w_2)$). By Lemma~8 in~\citep{belur2002}, the output cardinality of $\mathcal{C}_1^\text{can}\wedge_{(s,k)} \mathcal{C}_2^\text{can}$ can be determined from its latent variable representation as
\begin{align*}
\mathtt{p}(\mathcal{C}_1^\text{can}\wedge_{(s,k)} \mathcal{C}_2^\text{can})=\operatorname{rank}\begin{bmatrix}
L_1 & N_1\\ L_2 & N_2
\end{bmatrix}-\operatorname{rank}\begin{bmatrix}
N_1\\ N_2
\end{bmatrix}.
\end{align*}
Similarly, the output cardinality of $\mathcal{C}_1^\text{can}$ and $\mathcal{C}_2^\text{can}$ is given by
\begin{align*}
\mathtt{p}(\mathcal{C}_i^\text{can})=\operatorname{rank}\begin{bmatrix}
M_i & S_i & 0 & R_i\\ 0 & 0 & K_i & W_i
\end{bmatrix}-\operatorname{rank}\begin{bmatrix}
R_i\\ W_i
\end{bmatrix},\  i=1,2.
\end{align*}
It follows by~\eqref{eq:ch5LNmat} that $\mathtt{p}(\mathcal{C}_1^\text{can}\wedge_{(s,k)} \mathcal{C}_2^\text{can})$ is equal to
\begin{align} \label{eq:ch5pln}
\operatorname{rank}\begin{bmatrix}
L_1 & N_1\\ L_2 & N_2
\end{bmatrix}-\operatorname{rank}\begin{bmatrix}
R_1\\ W_1
\end{bmatrix}-\operatorname{rank}\begin{bmatrix}
R_2\\ W_2
\end{bmatrix}.
\end{align}
Hence, by \eqref{eq:ch5pln} and~\eqref{eq:ch5LNmat}, we find that
\begin{align*}
\mathtt{p}(\mathcal{C}_1^\text{can})+\mathtt{p}(\mathcal{C}_2^\text{can})&=\operatorname{rank}\begin{bmatrix}
L_1 & N_1
\end{bmatrix}+\operatorname{rank}\begin{bmatrix}
L_2 & N_2
\end{bmatrix}\\
&\quad-\operatorname{rank}\begin{bmatrix}
L_1 & N_1\\ L_2 & N_2
\end{bmatrix}+\mathtt{p}(\mathcal{C}_1^\text{can}\wedge_{(s,k)} \mathcal{C}_2^\text{can}).
\end{align*}
Therefore, $\mathtt{p}(\mathcal{C}_1^\text{can}\wedge_{(s,k)}\mathcal{C}_2^\text{can})=\mathtt{p}(\mathcal{C}_1^\text{can})+\mathtt{p}(\mathcal{C}_2^\text{can})$ if and only if~\eqref{eq:ch5propln} holds. This concludes the proof.
\end{proof}

Regularity of the interconnection of $\mathcal{C}_1^\text{can}$ and $\mathcal{C}_2^\text{can}$ turns out to be easily verifiable through regularity of the interconnections of subsystems $\mathcal{P}_1$ and $\mathcal{P}_2$ of the interconnected system that has to be controlled and of the interconnection of $\mathcal{K}_1$ and $\mathcal{K}_2$. We have the following result.
\begin{proposition}
The interconnection of $\mathcal{C}_1^\text{can}$ and $\mathcal{C}_2^\text{can}$ is regular if and only if the interconnection of $\mathcal{P}_1$ and $\mathcal{P}_2$ is regular and the interconnection of $\mathcal{K}_1$ and $\mathcal{K}_2$ is regular. That is, the interconnection of $\mathcal{C}_1^\text{can}$ and $\mathcal{C}_2^\text{can}$ is regular if and only if $\mathtt{p}(\mathcal{P}_1\wedge_{s} \mathcal{P}_2)=\mathtt{p}(\mathcal{P}_1)+\mathtt{p}(\mathcal{P}_2)$ and $\mathtt{p}(\mathcal{K}_1\wedge_{k} \mathcal{K}_2)=\mathtt{p}(\mathcal{K}_1)+\mathtt{p}(\mathcal{K}_2)$.
\end{proposition}
\begin{proof}
Let $R_i\left(\frac{\dd}{\dd t}\right)w_i+S_i\left(\frac{\dd}{\dd t}\right)s+M_i\left(\frac{\dd}{\dd t}\right)c_i=0$ be a minimal kernel representation for $\mathcal{P}_i$ and let $W_i\left(\frac{\dd}{\dd t}\right)w_i+K_i\left(\frac{\dd}{\dd t}\right)k=0$ be a minimal kernel representation for $\mathcal{K}_i$, $i=1,2$.

$(\Rightarrow)$ Assume that $\mathtt{p}(\mathcal{P}_1\wedge_{s} \mathcal{P}_2)=\mathtt{p}(\mathcal{P}_1)+\mathtt{p}(\mathcal{P}_2)$ and that $\mathtt{p}(\mathcal{K}_1\wedge_{k} \mathcal{K}_2)=\mathtt{p}(\mathcal{K}_1)+\mathtt{p}(\mathcal{K}_2)$. We then have that
\begin{align}
&\operatorname{rank}\begin{bmatrix}
R_1 & M_1 & 0 & 0 & S_1\\ 0 & 0 & R_2 & M_2 & S_2
\end{bmatrix}=\mathtt{p}(\mathcal{P}_1)+\mathtt{p}(\mathcal{P}_2)\nonumber\\
&=\operatorname{rank}\begin{bmatrix}
R_1 & M_1 & S_1
\end{bmatrix}+\operatorname{rank}\begin{bmatrix}
R_2 & M_2 & S_2
\end{bmatrix},\label{eq:ch5rankp}\\
&\operatorname{rank}\begin{bmatrix}
W_1 & 0 & K_1\\ 0 & W_2 & K_2
\end{bmatrix}=\mathtt{p}(\mathcal{K}_1\wedge_s \mathcal{K}_2)=\mathtt{p}(\mathcal{K}_1)+\mathtt{p}(\mathcal{K}_2)\nonumber\\
&=\operatorname{rank}\begin{bmatrix}
W_1 & K_1
\end{bmatrix}+\operatorname{rank}\begin{bmatrix}
W_2 & K_2
\end{bmatrix}. \label{eq:ch5rankk}
\end{align}
By Proposition~\ref{prop:ch5regcican}, $\mathcal{C}_1^\text{can}\wedge_{(s,k)} \mathcal{C}_2^\text{can}$ is regular if and only if~\eqref{eq:ch5propln} holds, i.e., if and only if
\begin{align} \label{eq:ch5rank}
\operatorname{rank}\begin{bmatrix}
A_1\\ B_1\\ A_2\\ B_2
\end{bmatrix}=\operatorname{rank}\begin{bmatrix}
A_1\\ B_1
\end{bmatrix}+\operatorname{rank}\begin{bmatrix}
A_2\\ B_2
\end{bmatrix},
\end{align}
with the sub-matrices $A_1:=\begin{bmatrix}
M_1\  0 \ S_1 \ 0 \ R_1  \ 0
\end{bmatrix}$, $A_2:=\begin{bmatrix}
0 \ M_2 \ S_2\  0\ 0\ R_2
\end{bmatrix}$, $B_1:=\begin{bmatrix}
0 \ 0 \  0 \  K_1 \  W_1 \  0
\end{bmatrix}$ and $B_2:=\begin{bmatrix}
0 \  0 \  0 \  K_2 \  0 \  W_2
\end{bmatrix}$.
Now, by~\eqref{eq:ch5rankp}, $A_1$ and $A_2$ do not have rows that are linearly dependent. Similarly, by~\eqref{eq:ch5rankk}, $B_1$ and $B_2$ do not have rows that are linearly dependent. Furthermore, $B_1$ and $A_2$ do not have rows that are linearly dependent and $A_1$ and $B_2$ do not have rows that are linearly dependent, by construction. Hence, $\left[\begin{smallmatrix}
A_1\\ B_1
\end{smallmatrix}\right]$ and $\left[\begin{smallmatrix}
A_2\\ B_2
\end{smallmatrix}\right]$ do not have rows that are linearly dependent. Therefore, \eqref{eq:ch5rank} holds true and it follows that $\mathcal{C}_1^\text{can}\wedge_{(s,k)} \mathcal{C}_2^\text{can}$ is regular.

$(\Leftarrow)$ Let $\mathcal{C}_1^\text{can}\wedge_{(s,k)} \mathcal{C}_2^\text{can}$ be regular. Then~\eqref{eq:ch5rank} holds true. But then $A_1$ and $A_2$ cannot contain dependent rows. Hence, $\mathtt{p}(\mathcal{P}_1\wedge_{s} \mathcal{P}_2)=\mathtt{p}(\mathcal{P}_1)+\mathtt{p}(\mathcal{P}_2)$. Moreover, by~\eqref{eq:ch5rank}, $B_1$ and $B_2$ cannot contain dependent rows. Hence, $\mathtt{p}(\mathcal{K}_1\wedge_{k} \mathcal{K}_2)=\mathtt{p}(\mathcal{K}_1)+\mathtt{p}(\mathcal{K}_2)$. This completes the proof.
\end{proof}

\section{Conclusions}
In this paper, we have considered the distributed control problem for linear interconnected systems in a behavioral setting. This setting allows to view distributed control from a more general perspective, where controllers are not intrinsically viewed as signal processors. Given a desired behavior represented by a linear interconnected system, the canonical distributed controller implements it, provided that necessary and sufficient conditions on the manifest behavior of the plant and desired behavior are satisfied. We have shown that regularity of the interconnections between subsystems in the plant and desired behavior are necessary and sufficient for regularity of the interconnections between subsystems in the canonical distributed controller.






\bibliographystyle{elsarticle-num-names}
\bibliography{rfrncs21a}

\end{document}